\documentclass{amsart}

\usepackage{amsmath}
\usepackage{amsthm}
\usepackage{amssymb}
\usepackage{amsfonts}
\usepackage{enumitem} % more enumerate/itemize options
\usepackage[bookmarks=true,hyperindex,pdftex,colorlinks,citecolor=blue]{hyperref}
\DeclareMathOperator{\lspan}{span}                          % linear span
\DeclareMathOperator{\conv}{conv}                           % convex hull
\DeclareMathOperator{\supp}{supp}                           % support
\DeclareMathOperator{\rad}{rad}                             % radius
\DeclareMathOperator{\Lip}{Lip}                             % Lipschitz functions
\newcommand{\NN}{\mathbb{N}}             % natural numbers
\newcommand{\RR}{\mathbb{R}}             % real numbers
\newcommand{\abs}[1]{\left|{#1}\right|}                     % absolute value
\newcommand{\pare}[1]{\left({#1}\right)}                    % parentheses
\newcommand{\set}[1]{\left\{{#1}\right\}}                   % set by extension
\newcommand{\norm}[1]{\left\|{#1}\right\|}                  % norm
\newcommand{\dual}[1]{{#1}^\ast}                            % dual or adjoint
\newcommand{\ddual}[1]{{#1}^{\ast\ast}}                     % double dual
\newcommand{\ball}[1]{B_{{#1}}}                             % unit ball
\newcommand{\duality}[1]{\left<{#1}\right>}                 % dual action
\newcommand{\cl}[1]{\overline{#1}}                          % closure
\newcommand{\wscl}[1]{\overline{#1}^{\dual{w}}}             % weak* closure
\newcommand{\weaks}{\textit{w}$^\ast$}                      % related to the weak* topology
\newcommand{\wsconv}{\stackrel{\dual{w}}{\longrightarrow}}  % weak* convergence
\newcommand{\lipfree}[1]{\mathcal{F}({#1})}                 % Lipschitz-free space
\newcommand{\Free}{\mathcal{F}}                             % Lipschitz-free space
\newcommand{\lipfreesub}[2]{\mathcal{F}_{#1}({#2})}         % Lipschitz-free subspace
\newcommand{\lipnorm}[1]{\norm{#1}_L}                       % Lipschitz norm/constant
\newcommand{\embd}{\delta}                                  % isometric embedding into free space
\newcommand{\idealsub}[2]{\mathcal{I}_{#1}({#2})}           % ideal of Lipschitz functions vanishing on a set (specifying the base space)
\newcommand{\mol}[1]{m_{#1}}                                % elementary molecules
\newcommand{\restricted}{\mathord{\upharpoonright}}

\theoremstyle{plain}
\newtheorem{theorem}{Theorem}[section]
\newtheorem{lemma}[theorem]{Lemma}

\newtheorem{proposition}[theorem]{Proposition}

\theoremstyle{definition}
\newtheorem*{definition*}{Definition}
\newtheorem{definition}[theorem]{Definition}

\theoremstyle{remark}
\newtheorem{remark}[theorem]{Remark}

\begin{document}

\title[Supports in $\lipfree{M}$ and applications]{Supports in Lipschitz-free spaces and applications to extremal structure}

\author[R. J. Aliaga]{Ram\'on J. Aliaga}
\address[R. J. Aliaga]{Instituto Universitario de Matem\'atica Pura y Aplicada, Universitat Polit\`ecnica de Val\`encia, Camino de Vera S/N, 46022 Valencia, Spain}
\email{raalva@upvnet.upv.es}

\author[E. Perneck\'a]{Eva Perneck\'a}
\address[E. Perneck\'a]{Faculty of Information Technology, Czech Technical University in Prague, Th\'akurova 9, 160 00, Prague 6, Czech Republic}
\email{perneeva@fit.cvut.cz}

\author[C. Petitjean]{Colin Petitjean}
\address[C. Petitjean]{LAMA, Univ Gustave Eiffel, UPEM, Univ Paris Est Creteil, CNRS, F-77447, Marne-la-Vall\'ee, France}
\email{colin.petitjean@u-pem.fr}

\author[A. Proch\'azka]{Anton\'in Proch\'azka}
\address[A. Proch\'azka]{Laboratoire de Math\'ematiques de Besan\c con,
Universit\'e Bourgogne Franche-Comt\'e,
CNRS UMR-6623,
16, route de Gray,
25030 Besan\c con Cedex, France}
\email{antonin.prochazka@univ-fcomte.fr}

\date{} % uncomment to remove date from title

% ABSTRACT

\begin{abstract}
We show that the class of Lipschitz-free spaces over closed subsets of any complete metric space $M$ is closed under arbitrary intersections, improving upon the previously known finite-diameter case. This allows us to formulate a general and natural definition of supports for elements in a Lipschitz-free space $\mathcal F(M)$. We then use this concept to study the extremal structure of $\mathcal F(M)$. We prove in particular that $(\delta(x) - \delta(y))/d(x,y)$ is an exposed point of the unit ball of $\mathcal F(M)$ whenever the metric segment $[x,y]$ is trivial, and that any extreme point which can be expressed as a finitely supported perturbation of a positive element must be finitely supported itself. We also characterise the extreme points of the positive unit ball: they are precisely the normalized evaluation functionals on points of $M$.
\end{abstract}

% KEYWORDS

%46B20 (Geometry and structure of normed linear spaces)
%46B04 (Isometric theory of Banach spaces)
%54E50 (Complete metric spaces)


\maketitle

% MAIN DOCUMENT

\section{Introduction}

The canonical preduals of Banach spaces of real-valued Lipschitz functions on metric spaces are known by a host of different names (Arens-Eells spaces \cite{Weaver2}, transportation cost spaces \cite{OsOs_2019}). In the Banach space geometer community, they are commonly referred to as Lipschitz-free spaces after Godefroy and Kalton coined the term in \cite{GoKa_2003}. 
Their most important application to non-linear geometry is likely their universal extension property: any vector-valued Lipschitz map defined on a given metric space extends uniquely as a bounded linear operator defined on the corresponding Lipschitz-free space (This result appears in disguise in the works of Kadets \cite[Corollary 2]{Kadets85} and Pestov~\cite[Theorem 1]{Pestov86} and was first explicitly stated by Weaver~\cite[Theorem 3.6]{Weaver2}, each with a different construction of the free space. For a short proof see \cite[Section 2]{CuDoWo_2016}.)
This allows us to turn a complicated (Lipschitz) mapping into a simple (linear) one at the expense of turning the metric domain into its more complex Lipschitz-free counterpart.
But this effort is only worthwhile if our knowledge of the structure of the Lipschitz-free spaces is deep enough. And indeed, although their definition looks simple, the current understanding of their structure is still quite limited. To name just a couple of open questions, it is not known whether the Lipschitz-free spaces over $\RR^n$ are isomorphic for different values of $n$ \cite{CuKaKa_2017}, or whether all Lipschitz-free spaces over subsets of $\RR^n$ have a Schauder basis \cite{HaPe_2014}; see Chapter 5 of \cite{GMZ} for other related open problems.

It is easy to deduce from the universal property that if $M$ is a metric space and $N$ is any subset containing the base point, then the Lipschitz-free space over $N$ is canonically identified with a subspace of the Lipschitz-free space over $M$. In a previous paper \cite{AlPe_2018} the first two authors showed that, when $M$ has finite diameter, the intersection of any family of such free spaces over closed subsets $K_i$ of $M$ is just the free space over the intersection of the sets $K_i$. This result seems quite intuitive, but its proof is rather nontrivial and depends on somewhat deep results by Weaver on the structure of algebraic ideals in the space of Lipschitz functions on $M$ (see Chapter 7 of \cite{Weaver2}).

In this note, we extend the result from \cite{AlPe_2018} to any complete metric space, thus showing that this natural property holds in general. The proof is not simpler; on the contrary, we reduce the general case to the known bounded case. To do this, we introduce a weighting operation on elements of the Lipschitz-free space using Lipschitz functions of bounded support.

In \cite{AlPe_2018}, a natural definition of support of an arbitrary element of a Lipschitz-free space was also proposed, and it was shown that such supports existed on any space that satisfied the intersection theorem. Thus, our first main result implies that the definition of support is valid in general. We develop its basic properties and obtain some equivalent characterizations of the concept.

Several applications of supports and weighted elements to the study of the extremal structure of Lipschitz-free spaces are also provided, adding to a number of other recent contributions \cite{AG,AlPe_2018, AlPePr_2019, Duality, Daugavet}. We start by showing that any elementary molecule $(\delta(x) - \delta(y))/d(x,y)$ such that the metric segment $[x,y]$ is trivial must be an exposed point of the unit ball of the Lipschitz-free space. This improves the main result in \cite{AlPe_2018}, which states that such an element must be an extreme point, and moreover provides a much shorter proof. We also prove that the extreme points of the positive unit ball are exactly the normalized evaluation functionals, i.e. $0$ and the elements $\delta(x)/d(x,0)$, and that they are all preserved. Finally, we prove that any extreme point of the form $\lambda+\mu$, where $\lambda$ is positive and $\mu$ has finite support, must be finitely supported. This provides some progress towards solving the conjecture that all extreme points of the ball of a Lipschitz-free space must be finitely supported.

\subsection{Notation}

Let us begin by introducing the notation that will be used throughout this paper. We will write $\ball{X}$ for the closed unit ball of a Banach space $X$ and $\ball{X}^O$ for its open unit ball. Next $M$ will denote a complete pointed metric space with metric $d$ and base point $0$, $B(p,r)$ will stand for the closed ball of radius $r$ around $p\in M$, and we will use the notation
\begin{align*}
d(p,A) &= \inf\set{d(x,p):x\in A} \\
\rad(A) &= \sup\set{d(x,0):x\in A}
\end{align*}
for $p\in M$ and $A\subset M$. These quantities will be called the \emph{distance} from $p$ to $A$ and the \emph{radius} of $A$, respectively.

Then $\Lip(M)$ will be the space of all real-valued Lipschitz functions on $M$, and $\Lip_0(M)$ will consist of all $f\in\Lip(M)$ such that $f(0)=0$. For $f\in\Lip(M)$ we will denote its Lipschitz constant by $\lipnorm{f}$. Recall that $\lipnorm{f}$ is a norm on $\Lip_0(M)$ that turns it into a dual Banach space. For any $x\in M$, we will use the notation $\embd(x)$ for the evaluation functional $f\mapsto f(x)$. Note that $\embd$ is a (non-linear) isometric embedding of $M$ into $\dual{\Lip_0(M)}$, and in fact $\lipfree{M}=\cl{\lspan}\,\embd(M)$ is the canonical predual of $\Lip_0(M)$, which we will call the \emph{Lipschitz-free space} over $M$. The weak$^\ast$ topology induced by $\lipfree{M}$ on $\Lip_0(M)$ coincides with the topology of pointwise convergence on norm-bounded subsets of $\Lip_0(M)$. In what follows the weak$^\ast$ topology will always be denoted \weaks.

We will say that $f\in\Lip(M)$ is positive if $f\geq 0$, i.e. if $f(x)\geq 0$ for all $x\in M$. Recall that the pointwise order is a partial order in $\Lip_0(M)$, and that a functional $\phi\in\lipfree{M}$ (or $\dual{\Lip_0(M)}$) is positive if $\phi(f)\geq 0$ for any positive $f\in\Lip_0(M)$. In that case, we will write $\phi\geq 0$; more generally, we will write $\phi\geq\psi$ whenever $\phi(f)\geq\psi(f)$ for all positive $f$.

Given a subset $K$ of $M$, we will also consider the subspace $\lipfreesub{M}{K}=\cl{\lspan}\,\embd(K)$ of $\lipfree{M}$ and the subspace $\idealsub{M}{K}$ of $\Lip_0(M)$ defined by
$$
\idealsub{M}{K}=\set{f\in\Lip_0(M):f(x)=0\text{ for all }x\in K}.
$$
In the above, if $M\neq \varnothing$, we adopt the the convention that $\lspan \varnothing=\set{0}$. Thus $\lipfreesub{M}{K}=\lipfreesub{M}{K\cup\set{0}}$
and $\idealsub{M}{K}=\idealsub{M}{K\cup\set{0}}$ for all $K \subset M$, and $\lipfreesub{M}{K}$ can be identified with the Lipschitz-free space $\lipfree{K\cup\set{0}}$. Let us also recall that $\lipfreesub{M}{K}^\perp=\idealsub{M}{K}$ and $\idealsub{M}{K}_\perp=\lipfreesub{M}{K}$. We refer to the monograph \cite{Weaver2} by Weaver for proofs of these and other basic facts and for further reference.

\section{The intersection theorem and supports}

Our first main result is the following:

\begin{theorem}
\label{tm:intersection_property}
Let $M$ be a complete pointed metric space and let $\{K_i:i\in I\}$ be a family of closed subsets of $M$. Then
$$
\bigcap_{i\in I}\lipfreesub{M}{K_i}=\mathcal{F}_{M}\pare{\bigcap_{i\in I}K_i} .
$$
\end{theorem}

Our proof will consist of reducing the problem to the case where $M$ is bounded, which was proved in \cite[Theorem 3.3]{AlPe_2018}. In order to do this, we will analyze the role of Lipschitz functions on $M$ with bounded support. Let us start by highlighting the following simple fact:

\begin{lemma}
\label{lm:bounded_supp_dense}
If $M$ is a pointed metric space, then the Lipschitz functions with bounded support are \weaks-dense in $\Lip_0(M)$ and in $\idealsub{M}{K}$ for any $K\subset M$.
\end{lemma}

\begin{proof}
For $r>0$, let $\Lambda_r\in\Lip(M)$ be the function defined by
%\begin{equation}
%\label{eq:lambda_r}
$$
\Lambda_r(x)=\begin{cases}
d(x,0) & \text{if } d(x,0)\leq r \\
2r-d(x,0) & \text{if } r\leq d(x,0)\leq 2r \\
0 & \text{if } 2r\leq d(x,0)
\end{cases} .
$$
%\end{equation}
This function is positive, has bounded support and satisfies $\lipnorm{\Lambda_r}\leq 1$. Moreover, for any $f\in\Lip_0(M)$ we have $\abs{f(x)}\leq\lipnorm{f}\cdot\Lambda_r(x)$ for any $x\in B(0,r)$. Thus, if we denote
$$
f_r(x)=\max\set{ \min\set{ f(x) , \lipnorm{f}\cdot\Lambda_r(x) } , -\lipnorm{f}\cdot\Lambda_r(x)}
$$
for $x\in M$, then $f_r\in\Lip_0(M)$ has bounded support, $\lipnorm{f_r}\leq \lipnorm{f}$, and $f_r(x)=f(x)$ for all $x\in B(0,r)$. It follows that $f_r\wsconv f$ as $r\rightarrow\infty$. Moreover, notice that $f_r(x)=0$ whenever $f(x)=0$, hence if $f\in\idealsub{M}{K}$ then $f_r\in\idealsub{M}{K}$. It follows that the Lipschitz functions with bounded support are \weaks-dense in $\idealsub{M}{K}$. In particular (taking $K=\set{0}$) they are \weaks-dense in $\Lip_0(M)$.
\end{proof}

Next, we show that pointwise multiplication with a Lipschitz function of bounded support always results in a Lipschitz function and, in fact, defines a continuous operator between Lipschitz spaces:

\begin{lemma}
\label{lm:multiplication_operator}
Let $M$ be a pointed metric space and let $h\in\Lip(M)$ have bounded support. Let $K\subset M$ contain the base point and the support of $h$. For $f\in\Lip_0(K)$, let $T_h(f)$ be the function given by
\begin{equation}
\label{eq:T_h}
T_h(f)(x)=\begin{cases}
f(x)h(x) & \text{if } x\in K \\
0 & \text{if } x\notin K
\end{cases} \,.
\end{equation}
Then $T_h$ defines a \weaks-\weaks-continuous linear operator from $\Lip_0(K)$ into $\Lip_0(M)$, and $\norm{T_h}\leq\norm{h}_\infty+\rad(\supp(h))\lipnorm{h}$.
\end{lemma}

\begin{proof}
Let $S=\supp(h)$. First, we show that $T_h(f)\in \Lip_0(M)$ for any $f\in\Lip_0(K)$. Clearly $T_h(f)(0)=0$. If $x,y\in S$, then 
\begin{align*}
\frac{\abs{T_h(f)(x)-T_h(f)(y)}}{d(x,y)} &= \frac{\abs{f(x)h(x)-f(y)h(y)}}{d(x,y)} \\
&\leq \frac{\abs{f(x)h(x)-f(x)h(y)}}{d(x,y)}+\frac{\abs{f(x)h(y)-f(y)h(y)}}{d(x,y)}\\
&\leq \sup_S\abs{f}\cdot\lipnorm{h}+\norm{h}_\infty\cdot\lipnorm{f} \\
&\leq (\rad(S)\lipnorm{h}+\norm{h}_\infty)\lipnorm{f} ,
\end{align*}
and if $x\in S$ and $y\in M\setminus S$ then
\begin{align*}
\frac{\abs{T_h(f)(x)-T_h(f)(y)}}{d(x,y)} &= \frac{\abs{f(x)h(x)-f(x)h(y)}}{d(x,y)} \\
&\leq \sup_S\abs{f}\cdot\lipnorm{h}\leq\rad(S)\lipnorm{h}\lipnorm{f} .
\end{align*}
Therefore the function $T_h(f)$ is Lipschitz and
$$
\lipnorm{T_h(f)}\leq \pare{\rad(S)\lipnorm{h}+\norm{h}_\infty}\cdot\lipnorm{f} .
$$
Hence, $T_h$ is a well defined and bounded operator from $\Lip_0(K)$ into $\Lip_0(M)$. Linearity is obvious. 

Finally, we prove that $T_h$ is \weaks-\weaks-continuous. By the Banach-Dieudonn\'e theorem, it suffices to show that it is \weaks-\weaks-continuous on bounded subsets of $\Lip_0(K)$. Since \weaks-convergence agrees with pointwise convergence on bounded subsets of Lipschitz spaces, it is enough to verify that $T_h(f_\gamma)\rightarrow T_h(f)$ pointwise whenever $f_{\gamma}\rightarrow f$ pointwise in $\ball{\Lip_0(K)}$, which is immediate from the definition of $T_h$.
\end{proof}

Let us make a few observations. Trivially $T_h(f)(x)=0$ whenever $x\in K$ is such that $f(x)=0$, hence $T_h$ maps $\idealsub{K}{L}$ into $\idealsub{M}{L}$ for any $L\subset K$. Moreover the function $T_h(f)$ does not depend on the choice of $K$, as long as it contains the support of $h$. Thus the requirement that $0\in K$ is not really a restriction, as one may always use the set $K\cup\set{0}$ instead.

Since $T_h$ is \weaks-\weaks-continuous, there is an associated bounded linear operator $W_h\colon\lipfree{M}\rightarrow\lipfree{K}$ such that $\dual{W_h}=T_h$. Thus we get the following consequence, which restricts to \cite[Lemma 3.1]{AlPe_2018} in the case where $K=M$ has finite diameter and $h(0)=0$:

\begin{proposition}
Let $M$ be a pointed metric space, let $h\in\Lip(M)$ have bounded support and let $K\subset M$ contain the base point and the support of $h$. Then for any $\mu\in\lipfree{M}$ we have $\mu\circ T_h\in\lipfree{K}$ and
$$
\norm{\mu\circ T_h}\leq (\norm{h}_\infty+\rad(\supp(h))\lipnorm{h})\cdot\norm{\mu} .
$$
Moreover, if $h\geq 0$ and $\mu\geq 0$ then $\mu\circ T_h\geq 0$.
\end{proposition}

\begin{proof}
Simply notice that $W_h(\mu)=\mu\circ T_h$ acts as a functional on $\Lip_0(K)$, since $\duality{W_h(\mu),f}=\duality{\mu,T_h(f)}$ for any $f\in\Lip_0(K)$. The inequality is immediate from Lemma \ref{lm:multiplication_operator}.
If $h\geq 0$ then $T_h$ takes positive functions into positive functions and the second statement follows.
\end{proof}

We now have all the tools we need to prove the main result of the section:

\begin{proof}[Proof of Theorem \ref{tm:intersection_property}]
Let $Y=\lspan\set{\idealsub{M}{K_i}:i\in I}$. We will show that $\wscl{Y}=\idealsub{M}{K}$ where $K=\bigcap_i K_i$. This is enough, as the annihilator relations imply then that
\begin{align*}
\bigcap_{i\in I}\lipfreesub{M}{K_i} &= \bigcap_{i\in I}\idealsub{M}{K_i}_\perp = \pare{\bigcup_{i\in I}\idealsub{M}{K_i}}_\perp \\
&= Y_\perp = \pare{\wscl{Y}}_\perp = \idealsub{M}{K}_\perp = \lipfreesub{M}{K} \,.
\end{align*}
The inclusion $\wscl{Y}\subset\idealsub{M}{K}$ is clear. For the reverse inclusion, take $f\in\idealsub{M}{K}$ and let $U$ be a \weaks-neighborhood of $f$ in $\Lip_0(M)$; it suffices to show that $U$ intersects $Y$.

We may assume that $f$ has bounded support by Lemma \ref{lm:bounded_supp_dense}. So let $S=\supp(f)$, define $h\in\Lip(M)$ by 
$$
h(x)=\max\set{1-d(x,S),0}
$$
for $x\in M$, and let $A=\supp(h)\cup\set{0}$. Thus $0\leq h\leq 1$, $h\restricted_S=1$ and $h=0$ outside of the bounded set $A$. 

Let $T_h\colon\Lip_0(A)\rightarrow\Lip_0(M)$ be as in \eqref{eq:T_h} and let $\hat{f}=f\restricted_A\in\Lip_0(A)$.
Note that $T_h(\hat f)=f$. 
Let $V$ be a $w^*$-neighborhood of $\hat{f}$ such that $T_h(V) \subset U$.
Since $A$ is bounded, we may apply \cite[Theorem 3.3]{AlPe_2018} to get 
$$\bigcap_{i\in I}\lipfreesub{A}{K_i\cap A}=\lipfreesub{A}{K\cap A},$$ and it follows that:
\begin{align*}
\idealsub{A}{K\cap A} = \pare{\bigcap_{i\in I}\lipfreesub{A}{K_i\cap A}}^\perp &= \pare{\pare{\bigcup_{i\in I}\idealsub{A}{K_i\cap A}}_\perp}^\perp \\
&= \wscl{\lspan}\set{\idealsub{A}{K_i \cap A}:i\in I} .
\end{align*}
Now $\hat{f}\in\idealsub{A}{K \cap A}$, so there must exist ${g}\in V$ of the form ${g}={g}_1+\ldots+{g}_n$ where ${g}_k\in\idealsub{A}{K_{i_k}\cap A}$, $i_k\in I$ for $k=1,\ldots,n$.
To complete the proof, note that $T_h(g_k)\in\idealsub{M}{K_{i_k}}$ for every $k=1,\ldots,n$ by the definition of $T_h$ and the comments below Lemma \ref{lm:multiplication_operator}. Hence $T_h(g)\in U\cap Y$.
\end{proof}

It was shown in \cite[Proposition 3.5]{AlPe_2018} that it was possible to define supports for elements of any Lipschitz-free space $\lipfree{M}$ that satisfied the property in Theorem \ref{tm:intersection_property}. We may therefore now state this definition in general:

\begin{definition}
Let $M$ be a pointed metric space and $\mu\in\lipfree{M}$. 
The \emph{support} of $\mu$, denoted $\supp(\mu)$, is the intersection of all closed subsets $K$ of $M$ such that $\mu\in\lipfreesub{M}{K}$.
\end{definition}

Let us mention some elementary properties of supports. 
To begin with, letting $\set{K_i}$ be as in Theorem~\ref{tm:intersection_property} the family of all closed subsets of $M$ such that $\mu\in\lipfreesub{M}{K_i}$, we get 
\begin{equation}\label{eq:mu_in_free_of_own_support}
\mu\in\lipfreesub{M}{\supp(\mu)}.
\end{equation}
In fact, \eqref{eq:mu_in_free_of_own_support} is an equivalent statement of Theorem \ref{tm:intersection_property} (see \cite[Proposition 3.5]{AlPe_2018}).
One observation is that the base point cannot be an isolated point of $\supp(\mu)$, as that would imply $\mu\in\lipfreesub{M}{K}$ where $K=\supp(\mu)\setminus\set{0}$ is closed. In particular, note that $\supp(0)=\varnothing$. This shows that supports are not completely stable under changes of base point, as e.g. $\supp(\delta(p))=\set{p}$ for $p\neq 0$ but changing the base point to $p$ converts $\delta(p)$ into $0$, with empty support. However, the discrepancy is limited to the new base point and only in the case where this point is isolated in the support.

Note also that if $\mu=\sum_n \mu_n$ where $\mu_n\in\lipfree{M}$, then it follows directly from the definition that $\supp(\mu)\subset\cl{\bigcup_n\supp(\mu_n)}$. The same happens if $\mu=\lim_n \mu_n$. In particular, by taking finitely supported $\mu_n$ it follows that $\supp(\mu)$ is always a closed separable subset of $M$.

We now describe several equivalent characterizations of supports:

\begin{proposition}
\label{pr:equiv_char_support}
Let $M$ be a complete pointed metric space, $K$ a closed subset of $M$, and $\mu\in\lipfree{M}$. Then the following are equivalent:
\begin{enumerate}[label={\upshape{(\roman*)}}]
\item $\supp(\mu)\subset K$,
\item $\mu\in\lipfreesub{M}{K}$,
\item $\duality{\mu,f}=\duality{\mu,g}$ for any $f,g\in\Lip_0(M)$ such that $f\restricted_K=g\restricted_K$.
\end{enumerate}
\end{proposition}

\begin{proof}
(i)$\Rightarrow$(ii): This is an immediate consequence of~\eqref{eq:mu_in_free_of_own_support}.

(ii)$\Rightarrow$(i): This follows trivially from the definition.

(ii)$\Leftrightarrow$(iii): Notice that (iii) is equivalent to $\duality{\mu,f-g}=0$ whenever $f-g$ vanishes in $K$, that is, to $\mu\in\idealsub{M}{K}_\perp$.
\end{proof}

The equivalence (i)$\Leftrightarrow$(iii) shows that, in particular
\begin{equation}
\label{eq:support_Th}
\supp(\mu\circ T_h)\subset\supp(\mu)\cap\supp(h)
\end{equation}
for any $\mu\in\lipfree{M}$ and $h\in\Lip(M)$ with bounded support. Indeed, if $f,g\in\Lip_0(M)$ coincide on $\supp(\mu)\cap\supp(h)$ then $T_h(f)$ and $T_h(g)$ coincide on $\supp(\mu)$ and thus 
$$
\duality{\mu\circ T_h,f}=\duality{\mu,T_h(f)}=\duality{\mu,T_h(g)}=\duality{\mu\circ T_h,g}.
$$
The inclusion in \eqref{eq:support_Th} may be strict. For instance, if $\supp(\mu)$ intersects $\supp(h)$ only at its boundary then $\mu\circ T_h=0$.

The following characterization of the support will also be used often:

\begin{proposition}
\label{pr:equiv_points_support}
Let $M$ be a complete pointed metric space and $\mu\in\lipfree{M}$. Then $p\in M$ lies in the support of $\mu$ if and only if for every neighbourhood $U$ of $p$ there exists a function $f\in\Lip_0(M)$ whose support is contained in $U$ and such that $\duality{\mu,f}\neq 0$. Moreover, in that case we may take $f\geq 0$.
\end{proposition}

\begin{proof}
Let $p\in M$. Assume that there exists a neighbourhood $U$ of $p$ such that for any function $f\in\Lip_0(M)$ with $\supp(f)\subset U$ we have $\duality{\mu,f}=0$. Take an open neighbourhood $V$ of $p$ for which $\overline{V}\subset U$. Then $\mu\in\idealsub{M}{M\setminus V}_{\perp}$ because every $f\in\idealsub{M}{M\setminus V}$ satisfies $\supp(f)\subset \overline{V}\subset U$. Hence $\mu\in\lipfreesub{M}{M\setminus V}$, so $\supp(\mu)\subset M\setminus V$ by the definition of $\supp(\mu)$ and $p\notin\supp(\mu)$.

On the other hand, suppose that $p\notin\supp(\mu)$ and let $U=M\setminus\supp(\mu)$. Then every $f\in\Lip_0(M)$ whose support is contained in $U$ obviously belongs to $\idealsub{M}{\supp(\mu)}=\lipfreesub{M}{\supp(\mu)}^{\perp}$. Therefore $\duality{\mu,f}=0$.

For the last statement, notice that $\duality{\mu,f}\neq 0$ implies that either $\duality{\mu,f^+}\neq 0$ or $\duality{\mu,f^-}\neq 0$.
\end{proof}

We finish this section by collecting some useful facts about positive elements of $\lipfree{M}$ and their supports:

\begin{proposition}
\label{pr:positive_facts}
Let $M$ be a complete pointed metric space and let $\mu$ and $\mu_n$, for $n\in\NN$, be positive elements of $\lipfree{M}$.
\begin{enumerate}[label={\upshape{(\alph*)}}]
\item $\norm{\mu}=\duality{\mu,\rho}$ where $\rho(x)=d(x,0)$.
\item $\norm{\sum_{n=1}^\infty\mu_n}=\sum_{n=1}^\infty\norm{\mu_n}$ whenever the last sum is finite.
\item If $f\in\Lip_0(M)$, $f\geq 0$ and $\duality{\mu,f}=0$, then $f=0$ on $\supp(\mu)$.
\item If $f\in\ball{\Lip_0(M)}$ and $\duality{\mu,f}=\norm{\mu}$, then $f=\rho$ on $\supp(\mu)$.
\end{enumerate}
\end{proposition}

\begin{proof}
(a) We have $\rho\in\ball{\Lip_0(M)}$ and any $f\in\ball{\Lip_0(M)}$ satisfies $f\leq\rho$, hence $\duality{\mu,f}\leq\duality{\mu,\rho}$.

(b) Evaluate $\sum_n\mu_n$ on $\rho$ and apply (a).

(c)
Suppose $f(p)>0$ for some $p\in\supp(\mu)$, so there are $c>0$ and $r>0$ such that $f\geq c$ in $B(p,r)$. By Proposition \ref{pr:equiv_points_support} there exists $h\in\Lip_0(M)$ such that $\supp(h)\subset B(p,r)$, $h\geq 0$ and $\duality{\mu,h}>0$. Scale $h$ by a constant factor so that $h\leq c$. Then $f-h\geq 0$ but $\duality{\mu,f-h}<0$, a contradiction.

(d) Apply (c) to the function $\rho-f$.
\end{proof}

\begin{proposition}
\label{pr:positive_support}
Let $M$ be a complete pointed metric space and let $\mu,\lambda$ be positive elements of $\lipfree{M}$. If $\mu\leq\lambda$ then $\supp(\mu)\subset\supp(\lambda)$.
\end{proposition}

\begin{proof}
Let $p\in\supp(\mu)$ and $U$ be a neighborhood of $p$. By Proposition \ref{pr:equiv_points_support} there exists $f\in\Lip_0(M)$ such that $\supp(f)\subset U$, $f\geq 0$ and $\duality{\mu,f}>0$. But then $\duality{\lambda,f}\geq\duality{\mu,f}>0$, so $p\in\supp(\lambda)$ applying Proposition \ref{pr:equiv_points_support} again.
\end{proof}

\section{Applications to extremal structure}

In this section, we develop some techniques based on supports and weighted elements to obtain new results related to the extremal structure of $\lipfree{M}$, in particular to analyze the extreme points of its unit ball and its positive unit ball. First, let us recall the definition of the extremal elements we will be considering:

\begin{definition}
Let $C$ be a convex subset of a Banach space $X$ and $x\in C$. We will say that $x$ is:
\begin{itemize}
\item an \emph{extreme point} of $C$ if it cannot be written as $x=\frac{1}{2}(y+z)$ with $y,z\in C\setminus\set{x}$,
\item an \emph{exposed point} of $C$ if there is $\dual{x}\in\dual{X}$ such that $\duality{x,\dual{x}}>\duality{y,\dual{x}}$ for any $y\in C\setminus\set{x}$,
\item a \emph{preserved extreme point} of $C$ if it is an extreme point of $\wscl{C}$ in $\ddual{X}$.
\end{itemize}
\end{definition}

Observe that exposed points and preserved extreme points are always extreme points. We will be considering the cases $C=\ball{\lipfree{M}}$ and $C=\ball{\lipfree{M}}^+$, the positive unit ball i.e. the set of positive elements of $\ball{\lipfree{M}}$. 

\subsection{Exposed molecules}

In the study of the extremal structure of Lipschitz-free spaces, a special role is played by the elements of the form
$$
\mol{pq}=\frac{\delta(p)-\delta(q)}{d(p,q)}
$$
for $p\neq q\in M$; note that $\norm{\mol{pq}}=1$. These elements are called \emph{elementary molecules}, sometimes just \emph{molecules}. One of the reasons for their relevance is the fact that any preserved extreme point of $\ball{\lipfree{M}}$ must be a molecule \cite[Corollary 3.44]{Weaver2}, which implies easily that any extreme point with finite support must also be a molecule. It is currently conjectured that all extreme points of $\ball{\lipfree{M}}$ must be molecules, but this has only been proved in a handful of cases, like subsets of $\RR$-trees \cite{AlPePr_2019} (including ultrametric spaces), and countable compacta. The latter follows from \cite[Theorem 2.1]{Dalet_2015} and \cite[Corollary 4.41]{Weaver2}, see also \cite[Corollary 4.2]{Duality} (note that the proof of \cite[Corollary 4.41]{Weaver2} included in the second edition of \cite{Weaver2} is flawed although it may be fixed; the weaker version contained in the first edition of the book is correct).

It is simply a matter of writing down the corresponding convex combination to see that $\mol{pq}$ can only be an extreme point of $\ball{\lipfree{M}}$ if the metric segment
$$
[p,q]=\set{x\in M:d(p,x)+d(x,q)=d(p,q)}
$$
only contains the points $p$ and $q$. The main result in \cite{AlPe_2018} states that this necessary condition is also sufficient. 
Here, we improve that result and show that any molecule satisfying this condition is actually an exposed point of $\ball{\lipfree{M}}$.
This also provides a significantly shorter proof of \cite[Theorem 1.1]{AlPe_2018}.

\begin{theorem}
\label{tm:exposed}
Let $M$ be a complete pointed metric space and $p\neq q\in M$. Then the following are equivalent:
\begin{enumerate}[label={\upshape{(\roman*)}}]
\item $\mol{pq}$ is an extreme point of $\ball{\lipfree{M}}$,
\item $\mol{pq}$ is an exposed point of $\ball{\lipfree{M}}$,
\item $[p,q]=\set{p,q}$.
\end{enumerate}
\end{theorem}

It is only necessary to prove the implication (iii)$\Rightarrow$(ii).
In our argument we will use the following fact. It can be found already in the proof of Theorem 2.37 in \cite{Weaver2}; see also page 89 therein. We include a short direct proof for the sake of completeness.

\begin{lemma}
\label{l:folklore}
Let $M$ be a metric space and 
$$\widetilde{M}=(M \times M) \setminus \set{(x,x):x\in M}.$$ 
Let us define $Q:\ell_1(\widetilde{M})\to \mathcal F(M)$ by $e_{(x,y)} \mapsto m_{xy}$ and extend it linearly on $\lspan\set{e_{(x,y)}}$.
Then $Q$ extends to an onto norm-one mapping which satisfies $\norm{\mu}=\inf\set{\norm{a}_1: Qa=\mu}$ for every $\mu \in \lipfree{M}$, i.e. $Q$ is a quotient map.
\end{lemma}

\begin{proof}
The fact that $\norm{Q}=1$ is clear so we can extend $Q$ to the whole space with the same norm. Let us call the extension $Q$ again. 
We will prove that $B_{\Free(M)}^O\subset Q(B_{\ell_1(\widetilde{M})}^O)$.
For this it is enough to use \cite[Lemma 2.23]{FHHMPZ}, i.e. we need to check that $B_{\mathcal F(M)}^O\subset \overline{Q(B_{\ell_1(\widetilde{M})}^O)}$.
But we have
$$
B_{\Free(M)}^O\subset B_{\Free(M)}=\overline{\conv}(V) \subset \overline{Q(B_{\ell_1(\widetilde{M})})}=\overline{Q(B_{\ell_1(\widetilde{M})}^O)} ,
$$
where $V=\set{m_{xy}:(x,y)\in \widetilde{M}}$ is the set of molecules of $\lipfree{M}$; note that $\ball{\lipfree{M}}=\cl{\conv}(V)$ follows from the fact that $V=-V$ is norming for $\Lip_0(M)$.
\end{proof}

\begin{remark}
We remark that if $\mu$ is an extreme point of $\ball{\lipfree{M}}$ such that $\mu=Qa$ for some $a\in\ball{\ell_1(\widetilde{M})}$, then $\mu$ must be a molecule. Indeed, suppose that $\mu=\sum_{n=1}^\infty a_n\mol{x_ny_n}$ where $\sum_{n=1}^\infty\abs{a_n}=1=\norm{\mu}$; without loss of generality, assume that each $a_n\geq 0$ and that $a_1>0$. If $a_1=1$ then clearly $\mu=\mol{x_1y_1}$. Otherwise $a_1\in(0,1)$ and we have
$$
\mu=a_1\mol{x_1y_1}+(1-a_1)\sum_{n=2}^\infty \frac{a_n}{1-a_1}\mol{x_ny_n}
$$
where the series on the right-hand side is in $\ball{\lipfree{M}}$ since the sum of the coefficients is 1, so extremality implies $\mu=\mol{x_1y_1}$ again.
\end{remark}

The main idea behind the proof of Theorem~\ref{tm:exposed} is that if some $f \in B_{\Lip_0(M)}$ exposes $m_{pq}$ in $B_{\lipfree{M}}$, then $f$ exposes $m_{pq}$ in particular among the molecules.
Such a candidate for the exposing functional is well known: it is the function $f_{pq}$ defined by
$$
f_{pq}(x)=\frac{d(p,q)}{2}\frac{d(x,q)-d(x,p)}{d(x,q)+d(x,p)}+C
$$
for $x\in M$, where the constant $C$ is chosen so that $f_{pq}(0)=0$.
This function was introduced and studied in \cite{ikw2}, where the following properties were proved:

\begin{lemma}\label{lemma:IKWfunction} Let $M$ be a complete metric space and let $p\neq q\in M$. We have
\begin{enumerate}
\item $f_{pq}$ is Lipschitz, $\norm{f_{pq}}_{L}=1$ and $\duality{\mol{pq},f_{pq}}=1$.
\item If $u\neq v\in M$ and $\varepsilon\geq 0$ are such that $\duality{\mol{uv},f_{pq}}\geq 1-\varepsilon$, then $u,v\in [p,q]_\varepsilon$ where
%\begin{equation}
%\label{eq:pqepsilon}
$$
[p,q]_\varepsilon=\set{x\in M: d(p,x)+d(x,q)\leq \frac{1}{1-\varepsilon}d(p,q)} .
$$
%\end{equation}
\item If $u\neq v \in M$ and $\frac{f_{pq}(u)-f_{pq}(v)}{d(u,v)}=1$, then $u,v\in [p,q]$.
\end{enumerate} 
\end{lemma}

Let us remark at this point that if $[p,q]=\set{p,q}$ then $f_{pq}$ exposes $m_{pq}$ among molecules (immediate from Lemma~\ref{lemma:IKWfunction}~(3)) and also among those $\mu \in B_{\Free(M)}$ with finite support (or more generally such that $\norm{\mu}=\norm{a}_1$ in the representation coming from Lemma~\ref{l:folklore}).

Using the concept of support, we can now prove the next strengthening of Lemma~\ref{lemma:IKWfunction}~(3).

\begin{lemma}
\label{lm:fpq_general}
Let $M$ be a complete pointed metric space and $p\neq q\in M$. If $\mu\in\ball{\lipfree{M}}$ is such that $\duality{\mu,f_{pq}}=1$, then $\supp(\mu)\subset [p,q]$.
\end{lemma}

\begin{proof}
Let $\delta,\varepsilon>0$. It follows from Lemma \ref{l:folklore} that we may find an expression $\mu=\sum_{n=1}^\infty a_n\mol{x_ny_n}$ where $x_n\neq y_n\in M$ for $n\in\NN$ and $\sum_n\abs{a_n}<1+\delta\varepsilon$. Let $I=\set{n\in\NN:\abs{\duality{\mol{x_ny_n},f_{pq}}}\geq 1-\varepsilon}$.
Then
\begin{align*}
1 = \duality{\mu,f_{pq}} &= \sum_{n=1}^\infty a_n\duality{\mol{x_ny_n},f_{pq}} \\
&= \sum_{n\in I} a_n\duality{\mol{x_ny_n},f_{pq}} + \sum_{n\in\NN\setminus I} a_n\duality{\mol{x_ny_n},f_{pq}} \\
&\leq \sum_{n\in I} \abs{a_n} + (1-\varepsilon)\sum_{n\in\NN\setminus I} \abs{a_n} \\
&< 1+\delta\varepsilon-\varepsilon\sum_{n\in\NN\setminus I} \abs{a_n} .
\end{align*}
Hence $\sum_{n\in\NN\setminus I}\abs{a_n}<\delta$, and it follows that
$$
\norm{\mu-\sum_{n\in I}a_n\mol{x_ny_n}}=\norm{\sum_{n\in\NN\setminus I}a_n\mol{x_ny_n}}\leq\sum_{n\in\NN\setminus I}\abs{a_n}<\delta .
$$
Notice that $x_n,y_n\in [p,q]_\varepsilon$ if $n\in I$, by Lemma \ref{lemma:IKWfunction}~(2).
Thus $\mu \in \lipfreesub{M}{[p,q]_\varepsilon}+\delta B_{\lipfree{M}}$.
Since $\delta$ was arbitrary, this shows that $\mu \in \lipfreesub{M}{[p,q]_\varepsilon}$.
But $\varepsilon>0$ was also arbitrary, so $\supp(\mu)\subset\bigcap_{\varepsilon>0} [p,q]_\varepsilon=[p,q]$.
\end{proof}

\begin{proof}[Proof of Theorem~\ref{tm:exposed}]
Assume (iii). We can assume without loss of generality that $0=q$. Indeed, a change of the base point in $M$ induces a linear isometry between the corresponding Lipschitz-free spaces which preserves the molecules. We will prove that $\mol{pq}$ is exposed by $f_{pq}$. Assume that $\mu\in B_{\lipfree{M}}$ is such that $\duality{\mu,f_{pq}}=1$. By Lemma \ref{lm:fpq_general}, $\mu$ must be supported on $[p,q]=\set{p,q}$, hence on $\set{p}$. Thus $\mu=\pm\mol{pq}$ but only the choice of the plus sign is reasonable. This proves (ii).
\end{proof}

\subsection{Extreme points of the positive ball}

Let us now consider the extreme points of $\ball{\lipfree{M}}^+$. We will characterize them and show that all of them are actually preserved. To achieve the latter, we require the following general fact about positive functionals on $\Lip_0(M)$:

\begin{lemma}
\label{lm:positive_functional_lemma}
Let $M$ be a pointed metric space and let $\mu,\lambda\in\dual{\Lip_0(M)}$ be such that $0\leq\mu\leq\lambda$. If $\lambda\in\lipfree{M}$, then $\mu\in\lipfree{M}$.
\end{lemma}

For the proof, let us recall that an element $\phi$ of $\dual{\Lip_0(M)}$ is \emph{normal} if it is such that $\duality{f_i,\phi}\rightarrow\duality{f,\phi}$ for any bounded net $(f_i)$ in $\Lip_0(M)$ that converges to $f$ pointwise and monotonically. Equivalently, $\phi$ is normal if $\duality{f_i,\phi}\rightarrow 0$ whenever $f_i\in\ball{\Lip_0(M)}$ and $f_i(x)$ decreases to $0$ for each $x\in M$. It is clear that every \weaks-continuous element of $\dual{\Lip_0(M)}$ is normal. Whether the converse holds is an open problem, but it was solved in the affirmative by Weaver for positive functionals \cite[Theorem 3.22]{Weaver2}.

\begin{proof}[Proof of Lemma \ref{lm:positive_functional_lemma}]
Let $(f_i)$ be a net such as stated above. Then we have $0\leq\duality{f_i,\mu}\leq\duality{f_i,\lambda}$ for each $i$. Since $\lambda$ is normal, $\duality{f_i,\lambda}\rightarrow 0$ and so $\duality{f_i,\mu}\rightarrow 0$ too. Hence $\mu$ is normal, and \cite[Theorem 3.22]{Weaver2} shows that it is \weaks-continuous.
\end{proof}

\begin{theorem}
Let $M$ be a complete pointed metric space. The extreme points of $\ball{\lipfree{M}}^+$ are precisely the normalized evaluation functionals, i.e. $0$ and $\embd(x)/d(x,0)$ for $x\in M\setminus\set{0}$. Moreover, all of them are preserved.
\end{theorem}

\begin{proof}
First we show that all extreme points are normalized evaluation functionals, or equivalently, their support does not contain more than one point. Let $\mu\in\ball{\lipfree{M}}^+$ with $\norm{\mu}=1$ be such that $\supp(\mu)$ contains at least two points $a$ and $b$; we will show that $\mu$ is not an extreme point of $\ball{\lipfree{M}}^+$. We may assume $a,b\neq 0$, since $0\in\supp(\mu)$ implies that it is an accumulation point, hence $\supp(\mu)$ is infinite.

Denote $r=d(a,b)/3$.
Let $h\in\Lip(M)$ be defined by
$$
h(x)=\max\set{1-\frac{d(x,B(a,r))}{r},0}
$$
so that $0\leq h\leq 1$, $h\restricted_{B(a,r)}=1$, $h\restricted_{B(b,r)}=0$, and $\supp(h)$ is bounded. Notice that $\mu\circ T_h\neq 0$. Indeed, let $f\in\Lip_0(M)$ such that $f\geq 0$, $f(a)=1$ and $\supp(f)\subset B(a,r)$, then $\duality{\mu\circ T_h,f}=\duality{\mu,f}>0$ by \eqref{eq:support_Th} and Proposition~\ref{pr:positive_facts}(c) because $a\in\supp(\mu)$. A similar argument using a function supported on $B(b,r)$ shows that $\mu\circ T_h\neq\mu$. Since $h$ and $1-h$ are both positive, so are $\mu\circ T_h$ and $\mu-\mu\circ T_h$ and thus
$$
\norm{\mu\circ T_h}+\norm{\mu-\mu\circ T_h}=\norm{\mu\circ T_h+(\mu-\mu\circ T_h)}=\norm{\mu}=1
$$
by Proposition \ref{pr:positive_facts}(b). But then
$$
\mu = \norm{\mu\circ T_h}\frac{\mu\circ T_h}{\norm{\mu\circ T_h}} + \norm{\mu-\mu\circ T_h}\frac{\mu-\mu\circ T_h}{\norm{\mu-\mu\circ T_h}}
$$
is a nontrivial convex combination of elements of $\ball{\lipfree{M}}^+$, as was to be shown.

Now let $x\in M$ and $\mu=\embd(x)/d(x,0)$ if $x\neq 0$ or $\mu=0$ if $x=0$; we will show that $\mu$ is really a preserved extreme point of $\ball{\lipfree{M}}^+$. Suppose that $\mu=\frac{1}{2}(\lambda+\nu)$ where $\lambda,\nu \in \wscl{\ball{\lipfree{M}}^+}$ are positive elements of $\ball{\dual{\Lip_0(M)}}$. Then $0\leq\frac{1}{2}\lambda\leq\mu$, so $\lambda\in\lipfree{M}$ by Lemma \ref{lm:positive_functional_lemma}. Moreover, Proposition \ref{pr:positive_support} implies that $\supp(\lambda)\subset\set{x}$. This is enough to conclude that $\lambda=\nu=\mu$, which finishes the proof.
\end{proof}

The fact that the normalized evaluation functionals are preserved extreme points of $\ball{\lipfree{M}}^+$ appears already in \cite[Corollary 7.36]{Weaver2}, although the result is stated only for bounded $M$. The reverse implication is new to the best of our knowledge.

Finally, let us note that $0$ is always an exposed point of $\ball{\lipfree{M}}^+$, but $\mol{x0}=\delta(x)/d(x,0)$ is exposed if and only if $[0,x]=\set{0,x}$. Indeed, one implication is immediate from Theorem \ref{tm:exposed}. The other follows from the fact that $f\in\ball{\Lip_0(M)}$ norms $\mol{x0}$ if and only if $f(x)=d(x,0)$, but then $f(y)=d(y,0)$ for any $y\in[0,x]$ so $f$ norms $\mol{y0}$ too.

\subsection{Extreme points which are almost positive}

As a final application, let us analyze the extreme points of $\ball{\lipfree{M}}$ that may be expressed as a finitely supported perturbation of a positive element of $\lipfree{M}$. We will prove that these extreme points must have finite support and hence be elementary molecules.

Let $S$ be a non-empty subset of $M$. For $f\in\Lip(S)$ with $\lipnorm{f}\leq 1$ and $x\in M$ we denote
\begin{equation}
\label{eq:inf_extension}
f_I(x):=\inf_{q \in S} \,(f(q)+d(q,x)) .
\end{equation}
Then $f_I$ is an extension of $f$ to $M$ such that $\lipnorm{f_I}\leq 1$. In fact, it is the largest 1-Lipschitz extension in the following sense: for every $x\in M\setminus S$, $(f_I)\restricted_{S\cup\set{x}}$ is the largest 1-Lipschitz extension of $f$ to $S \cup \set{x}$. In other words, if $g$ is an extension of $f$ to $M$ such that $\lipnorm{g}\leq 1$ then $g\leq f_I$.

We require the following simple observation:

\begin{lemma}
\label{lm:norm_formula}
Let $M$ be a complete pointed metric space and $\mu,\lambda\in\lipfree{M}$ such that $\lambda\geq 0$. Let $S=\supp(\mu)\cup\set{0}$ and define
$$
N(f)=\duality{\mu+\lambda,f_I}
$$
for $f\in\ball{\Lip_0(S)}$, where $f_I$ is defined by \eqref{eq:inf_extension}. Then $N$ is a concave function that attains its maximum on $\ball{\Lip_0(S)}$, and 
\[\max_{f\in\ball{\Lip_0(S)}}N(f)=\norm{\mu+\lambda}.\]
\end{lemma}

\begin{proof}
It is obvious that $N(f)\leq\norm{\mu+\lambda}$ for any $f\in\ball{\Lip_0(S)}$.
By the Hahn-Banach theorem, there is $g\in\ball{\Lip_0(M)}$ such that $\norm{\mu+\lambda}=\duality{\mu+\lambda,g}$.
Let $f=g\restricted_S$, then $f\in\ball{\Lip_0(S)}$ and $f_I\geq g$, so $\duality{\lambda,f_I}\geq\duality{\lambda,g}$.
Moreover, $(f_I)\restricted_S=g\restricted_S$ and hence $\duality{\mu,f_I}=\duality{\mu,g}$ by Proposition \ref{pr:equiv_char_support}.
It follows that
$$
N(f)=\duality{\mu,f_I}+\duality{\lambda,f_I}\geq\duality{\mu,g}+\duality{\lambda,g}=\norm{\mu+\lambda} .
$$

To show that $N$ is concave, note that this is equivalent to the map $f\mapsto\duality{\lambda,f_I}$ being concave, i.e. to
$$
\duality{\lambda,(cf+(1-c)g)_I} \geq c\duality{\lambda,f_I}+(1-c)\duality{\lambda,g_I}
$$
for any $f,g\in\ball{\Lip_0(S)}$ and $c\in(0,1)$.
Since $\lambda\geq 0$, it suffices to show that
$$
(cf+(1-c)g)_I\geq cf_I+(1-c)g_I
$$
pointwise, that is
\begin{multline*}
\inf_{q\in S}\big(cf(q)+(1-c)g(q)+d(x,q)\big) \geq \\ c\cdot\inf_{q\in S}\big(f(q)+d(x,q)\big)+(1-c)\cdot\inf_{q\in S}\big(g(q)+d(x,q)\big)
\end{multline*}
for every $x\in M$. But this is obvious.
\end{proof}

\begin{theorem}
\label{tm:positive_plus_finite}
Let $M$ be a complete pointed metric space and $\lambda,\mu\in\lipfree{M}$ such that $\lambda\geq 0$ and $\mu$ has finite support.
If $\lambda+\mu$ is an extreme point of $\ball{\lipfree{M}}$, then it has finite support.
\end{theorem}

\begin{proof}
Let $S=\supp(\mu)\cup\set{0}$, and consider the function $N\colon\lipfree{M}^+\times\ball{\Lip_0(S)}\rightarrow\RR$ given by $N(\lambda,f)=\duality{\lambda+\mu,f_I}$.
Denote also $N_\lambda(f)=N(\lambda,f)$.
By Lemma \ref{lm:norm_formula}, $\norm{\lambda+\mu}$ is the maximum of $N_\lambda(f)$ for $f\in\ball{\Lip_0(S)}$, and $N_\lambda$ is a concave function for fixed $\lambda$.
Moreover, it is easy to verify directly that $N_\lambda$ is continuous using the boundedness of $S$.
It follows from concavity that $N_\lambda(f)=\norm{\lambda+\mu}$ if and only if $f$ is a local maximum of $N_\lambda$, i.e. if and only if
$$
\duality{\lambda,(f+g)_I-f_I}\leq\duality{\mu, f_I-(f+g)_I}=\duality{-\mu,g}
$$
for all $g\in\Lip_0(S)$ in a neighborhood of $0$ such that $f+g\in\ball{\Lip_0(S)}$.

Suppose now that $\lambda$ has infinite support, and let $f \in \ball{\Lip_0(S)}$ be such that $\norm{\lambda+\mu}=N_\lambda(f)$.
We will show that there is a nonzero $v\in\lipfree{M}$ such that $\lambda\pm v\geq 0$, $\duality{v,f_I}=0$, and $\duality{v,(f+g)_I-f_I}=0$ for all $g\in\Lip_0(S)$ in a neighborhood of $0$.
The argument above will then imply that
$$
\norm{\lambda\pm v+\mu}=N(\lambda\pm v,f)=N(\lambda,f)\pm\duality{v,f_I}=N(\lambda,f)=\norm{\lambda+\mu}
$$
so $\lambda+\mu$ cannot be an extreme point of $\ball{\lipfree{M}}$.
Thus, if $\lambda+\mu$ is extreme then it must be finitely supported.

For every non-empty subset $K\subset S$, define the set
$$
A_K=\set{x\in M: f_I(x)=f(q)+d(x,q) \text{ if and only if } q\in K} .
$$
That is, $A_K$ contains those points $x\in M$ where the infimum in the definition of $f_I(x)$ is attained exactly for all $q\in K$ and nowhere else.
Since $S$ is finite, the sets $A_K$ form a finite partition of $M$.
Choose $K$ of the smallest possible cardinality such that $\supp(\lambda)\cap A_K$ contains at least three points $p_1,p_2,p_3$.
Let
%\begin{multline*}
%\varepsilon=\frac{1}{4}\min \left\{ \big(f(q')+d(p_i,q')\big)-\big(f(q)+d(p_i,q)\big) : \right. \\
%\left. q\in K,q'\in S\setminus K,i=1,2,3 \right\}
%\end{multline*}
$$
\varepsilon=\frac{1}{4}\min \set{ \big(f(q')+d(p_i,q')\big)-\big(f(q)+d(p_i,q)\big) : q\in K,q'\in S\setminus K,i=1,2,3 }
$$
and choose $r\in (0,\varepsilon)$ such that the balls $B(p_i,r)$ are disjoint, do not contain the base point, and do not intersect the finite sets $\supp(\lambda)\cap A_L$ for any $L\subsetneq K$.
By Proposition \ref{pr:equiv_points_support}, for $i=1,2,3$ there exist non-negative functions $h_i\in\Lip_0(M)$ supported on $B(p_i,r)$ such that $\duality{\lambda,h_i}>0$.
Now choose real constants $c_1,c_2,c_3$, not all of them equal to zero, such that
\begin{alignat*}{7}
& c_1\duality{\lambda,h_1} &&+ c_2\duality{\lambda,h_2} &&+ c_3\duality{\lambda,h_3} &&= 0 \\
& c_1\duality{\lambda,h_1\cdot f_I} &&+ c_2\duality{\lambda,h_2\cdot f_I} &&+ c_3\duality{\lambda,h_3\cdot f_I} &&= 0
\end{alignat*}
and $\abs{c_i}\leq 1/\norm{h_i}_\infty$.
Let $h=c_1h_1+c_2h_2+c_3h_3$ and $v=\lambda\circ T_h$.

Let us check that $v$ satisfies the required conditions.
By construction, we have $\duality{\lambda,h}=0$ and $\duality{v,f_I}=\duality{\lambda,h\cdot f_I}=0$.
Also,
$$
\duality{\lambda\pm v,g}=\duality{\lambda,g\pm T_h(g)}=\duality{\lambda,g\cdot(1\pm h)}
$$
for any $g\in\Lip_0(M)$.
By the choice of $c_i$ we have $1\pm h\geq 0$ and so $\duality{\lambda\pm v,g}\geq 0$ whenever $g\geq 0$, that is, $\lambda\pm v\geq 0$.
Also, choose $i\in\set{1,2,3}$ such that $c_i\neq 0$, then there is $\varphi\in\Lip_0(M)$ such that $\varphi=1$ on $B(p_i,r)$ and $\varphi=0$ on $\supp(h)\setminus B(p_i,r)$, hence $\duality{v,\varphi}=c_i\duality{\lambda,h_i}\neq 0$. This shows that $v\neq 0$.

Finally, let $x\in\supp(v)$. Then $x\in\supp(\lambda)\cap\supp(h)$ by \eqref{eq:support_Th}, so there is $i\in\set{1,2,3}$ such that $x\in B(p_i,r)$. Therefore, if $q\in K$, $q'\in S\setminus K$ then 
\begin{align*}
f(q')+d(x,q') &\geq f(q')+d(p_i,q')-d(x,p_i) \\
&\geq f(q)+d(p_i,q)+4\varepsilon-d(x,p_i) \\
&\geq f(q)+d(x,q)+4\varepsilon-2d(x,p_i) \\
&\geq f(q)+d(x,q)+2\varepsilon
\end{align*}
and so $x\in A_L$ for some $L\subset K$, hence $x\in A_K$ by construction.
If we now take any $g\in\Lip_0(S)$ such that $\norm{g}_\infty<\varepsilon$ and $\lipnorm{f+g}\leq 1$, then
$$
f(q')+g(q')+d(x,q')>f(q)+g(q)+d(x,q)
$$
for any $q\in K$, $q'\in S\setminus K$, and it follows that $(f+g)_I(x)=f_I(x)+\gamma$ where $\gamma=\min_{q\in K} g(q)$.
Thus we get
$$
\duality{v,(f+g)_I-f_I}=\duality{\lambda,h\cdot\gamma}=\gamma\cdot\duality{\lambda,h}=0 .
$$
This completes the proof.
\end{proof}

\section*{Acknowledgments}

This research was carried out during visits of the first and the second author to the Laboratoire de Math\'ematiques de Besan\c{c}on in 2019. Both authors are grateful for the opportunity and the hospitality.

This work was supported by the French ``Investissements d'Avenir'' program, project ISITE-BFC (contract ANR-15-IDEX-03). R. J. Aliaga was also partially supported by the Spanish Ministry of Economy, Industry and Competitiveness under Grant MTM2017-83262-C2-2-P. E. Perneck\'a was supported by the grant GA\v CR 18-00960Y of the Czech Science Foundation.

% BIBLIOGRAPHY

\end{document}